\documentclass[12pt,onecolumn,draftcls]{IEEEtran} 

\IEEEoverridecommandlockouts
\usepackage{graphicx}
\usepackage{amsmath}
\usepackage{amssymb}
\usepackage{comment}
\usepackage{color}

\usepackage{ifthen}

\newboolean{arxivbool}
\setboolean{arxivbool}{false}
\setboolean{arxivbool}{true}

\newtheorem{nntheorem}{\bf Theorem}
\newtheorem{nnassumption}{\bf Assumption}
\newtheorem{nndefinition}{\bf Definition}[section]
\newtheorem{nnlemma}[nndefinition]{\bf Lemma}
\newtheorem{nncorollary}[nndefinition]{\bf Corollary}
\newtheorem{nnproposition}[nndefinition]{\bf Proposition}
\newtheorem{nexample}[nndefinition]{\bf Example}

\newenvironment{theorem}
{\begin{nntheorem}\it}
{\end{nntheorem}}
\newenvironment{proposition}
{\begin{nnproposition}\it}
{\end{nnproposition}}

\newenvironment{corollary}
{\begin{nncorollary}\it}
{\end{nncorollary}}

\newenvironment{nnexample}
{\begin{nexample}\rm}{\end{nexample}}

\newtheorem{nnremark}[nndefinition]{\bf Remark}
\newenvironment{remark}{\begin{nnremark}}{\hfill \hspace*{1pt}\hfill $\circ$\end{nnremark}}

\definecolor{marron}{rgb}{0.64,0.16,0.16}

\ifthenelse{\boolean{arxivbool}}{\newcommand{\startchris}{\color{black}}}
{\newcommand{\startchris}{\color{red}}}
\newcommand{\stopchris}{\color{black}}

\ifthenelse{\boolean{arxivbool}}{\newcommand{\heightfigure}{8}}{\newcommand{\heightfigure}{5}}

\def\ker{\mathtt{Ker}}

\def\RR{\mathbb{R}}
\def\NN{\mathbb{N}}
\def\diag{\mathtt{diag}}

\ifthenelse{\boolean{arxivbool}}{\title{Stability of switched linear hyperbolic systems by Lyapunov techniques (full version)}}
{\title{Stability of switched linear hyperbolic systems by Lyapunov techniques}}

\author{Christophe Prieur, Antoine Girard, and Emmanuel Witrant\thanks{C. Prieur and E. Witrant are from the Department of Automatic Control,  Gipsa-lab, Universit\'e de Grenoble, 11 rue des Math\'ematiques, BP 46,
38402 Saint-Matin d'H\`eres Cedex, France.  Email: {\tt\small $\lbrace$christophe.prieur, emmanuel.witrant$\rbrace$@gipsa-lab.fr} and A. Girard is with Laboratoire Jean Kuntzmann, Universit\'e de Grenoble, BP 53, 38041 Grenoble, France, 
{\tt\small antoine.girard@imag.fr}. This work is partly supported by HYCON2 Network of Excellence Highly-Complex and Networked Control Systems, grant agreement 257462.}}

\begin{document}
\maketitle

\begin{abstract}
Switched linear hyperbolic partial differential equations are considered in this
paper. They model infinite dimensional systems of conservation laws and balance
laws, which are potentially affected by a distributed source or sink term.
The dynamics and the boundary conditions are subject to abrupt changes given by 
a switching signal, modeled
as a piecewise constant function and possibly a dwell time.
By means of Lyapunov techniques some sufficient
conditions are obtained for the exponential stability of the
switching system, uniformly for all switching signals.
Different cases are considered with or without a dwell time
assumption on the switching signals, and on the number of 
positive characteristic velocities (which may also depend on the
switching signal). Some numerical simulations are also given to
illustrate some main results, and to motivate this study.
\end{abstract}


\section{Introduction}
Lyapunov techniques are commonly used for the stability analysis of dynamical systems, {such} as those modeled by partial differential 
equations (PDEs). 
The present paper focuses on a class of {one-dimensional}  
hyperbolic equations that describe, for example, systems of conservation laws or balance laws (with a source term),
see \cite{DiagneBastinCoron:automatica:11}.

\ifthenelse{\boolean{arxivbool}}{A switching behavior occurs for many control applications when the evolution processes involve logical decisions, see 
\cite{ElFarraChristofides:chemical:2004} for the case {where} a stabilizing feedback is designed by means of Lyapunov techniques applied to a 
discretization of 
switched parabolic PDE; see also \cite{HanteLeugeringSeidman:applied:2009}, where the well-posed issue and the dependence of the solutions
on the data of a network of hyperbolic equations with switching as a control are considered. Switching can indeed be an efficient control strategy 
for many 
infinite dimensional systems such as the wave equation (\cite{GugatTucsnak:string:2011}), the heat equation (\cite{Zuazua:Switching:2011}) or 
other infinite dimensional systems written in abstract form (as in \cite{HanteSigalottiTucsnak:arxiv:2011}).}{}

\startchris The exponential stabilizability of such systems is often proved by means of a Lyapunov function, 
as illustrated by the 
contributions from 
\cite{KrsticSmyshlyaev:scl:08,PrieurHalleux04} 
where different control problems are solved for particular 
hyperbolic equations. 
For more general nonlinear hyperbolic equations, the knowledge of Lyapunov functions 
can be useful for the stability analysis of a system of conservation laws
(see \cite{CoronAndreaBastin07}), or even for the design of exponentially
stabilizing boundary controls (see \cite{CoronBastinAndrea:SIAM:08}). 
Other control techniques may be useful, such as Linear Quadratic regulation \cite{Aksikas:automatica:2009} or semigroup theory \cite[Chap. 6]{luo-guo-morgul-1999}.
\stopchris

In this paper, the class of hyperbolic systems of balance laws is first considered without any switching rule and we state sufficient conditions to 
derive a Lyapunov function for this class of systems. It allows us to relax \cite{DiagneBastinCoron:automatica:11} where the Lyapunov stability for hyperbolic systems of balance laws has been first tackled (see also 
\cite{CoronAndreaBastin07}). Then, switched systems are considered and sufficient conditions for the asymptotic stability of a class of linear hyperbolic systems with switched dynamics and switched boundary conditions are stated. Some stability conditions depend on the average dwell
time of the switching signals (if such a positive dwell time does exist).
The stability property
depends on the classes of the switching rules applied to the dynamics (as in \cite{liberzon:book:2003} for finite dimensional systems). The present paper is also related to \cite{PrieurMazenc:hyper:11} where unswitched time-varying hyperbolic systems are considered.

\startchris
In \cite{AminHanteBayen:IEEETAC:2011}, the condition of \cite{libook:94} is employed. It allows analyzing the stability of hyperbolic systems, assuming a stronger hypothesis on the boundary conditions. More precisely, our approach generalizes the condition of \cite{CoronAndreaBastin07}, which is known to be strictly weaker than the one of \cite{libook:94}. Therefore our stability conditions are strictly weaker than the ones of \cite{AminHanteBayen:IEEETAC:2011}.
Moreover the technique in \cite{AminHanteBayen:IEEETAC:2011} is trajectory-based via the method of characteristics, while our approach is based on Lyapunov functions, allowing for numerically tractable conditions.
\stopchris
Indeed, the obtained sufficient conditions are written in terms of matrix inequalities, which can be solved numerically. Furthermore the estimated speed of exponential convergence is provided and can be optimized. See \ifthenelse{\boolean{arxivbool}}{Section \ref{sec:comp}}{\cite{Prieur:Girard:Witrant:arxiv:2013}} for the use of line search algorithms to numerically compute the variables in our stability conditions, and thus to compute Lyapunov functions.
The main results and the computational aspects are illustrated on two examples of switched linear hyperbolic  systems.

\ifthenelse{\boolean{arxivbool}}{The paper is organized as follows. The class of switched linear hyperbolic systems of balance laws considered in this paper is given in Section \ref{sec:2} and a first stability condition is proven. Switched systems of balance laws are presented in Section \ref{sec:3}. In Section \ref{sec:4} our main results are derived for the stability of switched hyperbolic systems. The conditions depend on the class of piecewise constant switching signals that is considered (with and without a sufficiently large dwell time). The stability conditions may also differ if the number of positive characteristic velocities does not depend on the switching signal (see Section \ref{sec:4.1}) or if this number is a function of this signal (see Section \ref{sec:stabedep}).\ifthenelse{\boolean{arxivbool}}{ Section \ref{sec:comp} collects the discussions on computational aspects. It deals in particular with the numerical check of our
stability conditions, and the numerical computations of the considered Lyapunov functions.}{}
In Section \ref{sec:ex} two examples illustrate the main results and motivate the class of Lyapunov functions considered in this paper.
}{}

\ifthenelse{\boolean{arxivbool}}{}{Due to space limitation, some proofs have been omitted and collected in \cite{Prieur:Girard:Witrant:arxiv:2013}.}

{\small {\bf Notation.} 
The set $\RR_+$ is the set of nonnegative real numbers. 
Given a matrix $G$, the transpose matrix of $G$ is denoted as $G^\top$.
When $G$ is invertible, then, to simplify the notation, $(G^{-1})^\top$ is denoted as $G^{-\top}$. For positive integers $m$ and $n$, $I_n$ and $0_{n,m}$ are respectively the identity and the null matrix in $\RR^{n\times n}$ and in $\RR^{n\times m}$. Given some scalar values $(a_1, \ldots , a_n)$, $\diag(a_1, \ldots , a_n)$ is the matrix in $\RR^{n\times n}$ with zero non-diagonal entries, and with $(a_1, \ldots , a_n)$ on the diagonal. 
Moreover given two matrices $A$ and $B$, $\diag [A,B]$ is the block diagonal matrix formed by $A$ and $B$ (and zero for the other entries).  
The notation $A\geq B$ means that $A-B$ is positive semidefinite.
The usual Euclidian norm in $\RR^n$ is denoted by $|\cdot |$ and the associated matrix norm is denoted $\|\cdot \|$, whereas the set of all functions $\phi:(0,1)\rightarrow \RR^n$ such that $\int_0^1 | \phi (x)| ^2 <\infty $ is denoted by $L^2 ((0,1);\RR^n)$ that is equipped with the norm $\|\cdot \|_{L^2 ((0,1);\RR^n)}$. Given a topological set $\mathcal{S}$, and an interval $I$ in $\RR_+$,  
the set $C^0(I,\mathcal{S})$ is the set of continuous functions $\phi:I\rightarrow \mathcal{S}$.
}

\section{Linear hyperbolic systems}
\label{sec:2}
Let us first consider the following linear hyperbolic partial differential equation:
\begin{equation}
\label{eq:lin}
\partial_t y(t,x)+\Lambda  \partial_ x
  y(t,x) = Fy(t,x), \quad  x \in [0,1],\; t\in \RR_+
\end{equation}
where $y:\RR_+\times [0,1] \rightarrow \RR^n$, $F$ is a matrix in $\RR^{n\times n}$, $\Lambda$ is a diagonal matrix in $\RR^{n\times n}$ such that
$\Lambda=\diag(\lambda_{1}, \ldots , \lambda_{n})$, with
$\lambda_k <0$ for $k\in\{1,\dots,m\}$ and $\lambda_k>0$ for $k\in \{m+1,\dots,n\}$.
We use the notation
$y=\left(\begin{smallmatrix} y^-\\y^+\end{smallmatrix}\right)$, where $y^-:\RR_+\times [0,1] \rightarrow\RR^{m}$ and $y^+:\RR_+\times [0,1] \rightarrow\RR^{n-m}$.
In addition, we consider the following boundary conditions:
 \begin{equation}\label{eq:bound}
  \left(\begin{smallmatrix}
      y^-(t,1) \\
      y^+(t,0) \\
    \end{smallmatrix}
  \right)
  =G
  \left(\begin{smallmatrix}
      y^-(t,0) \\
      y^+(t,1) \\
    \end{smallmatrix}
  \right), \quad   t \in \RR_+
\end{equation}
where $G$ is a matrix in $\RR^{n\times n}$. Let us introduce the matrices $G_{--}$ in $\RR^{m\times m}$,  $G_{-+}$ in $\RR^{m\times (n-m)}$,  
$G_{+-}$ in $\RR^{(n-m)\times m }$ and $G_{++ }$ in  $\RR^{(n-m)\times (n-m)}$ such that $G=\left(\begin{smallmatrix}
G_{--} &G_{-+}\\
G_{+-} &G_{++} \end{smallmatrix}\right)$.

We shall consider an initial condition given by 
\begin{equation}
\label{eq:init}
y(0,x)=y^0(x), \quad  x \in (0,1)
\end{equation}
where $y^0\in L^2((0,1);\RR^n)$.
Then, it can be shown (see e.g.~\cite{DiagneBastinCoron:automatica:11}) that there exists a unique solution 
$y \in C^0(\RR_+;L^2((0,1);\RR^n))$ to the initial value problem (\ref{eq:lin})-(\ref{eq:init}). 
As these solutions may not be differentiable everywhere, the concept of weak solutions of partial differential equations has to be used 
(see again~\cite{DiagneBastinCoron:automatica:11} for more details).
The linear hyperbolic system (\ref{eq:lin})-(\ref{eq:bound}) is said to be {\em globally exponentially stable} (GES) if there exist $\nu >0$ and
$C>0$ such that, for every $y_0\in L^2((0,1);\RR^n)$; the solution to the initial value problem (\ref{eq:lin})-(\ref{eq:init}) satisfies
\begin{equation}
\label{eq:stab}
\|y(t,.)\|_{L^2((0,1);\RR^n)} \le C e^{-\nu t} \|y^0\|_{L^2((0,1);\RR^n)},  \quad \forall t\in \RR_+ .
\end{equation}
Sufficient conditions for exponential stability of (\ref{eq:lin})-(\ref{eq:init}) have been obtained in~\cite{DiagneBastinCoron:automatica:11}
using a Lyapunov function. In this section, we present an extension of this result. This extension will be also useful for subsequent work on
switched linear hyperbolic systems.
  
Let $\Lambda^+=\diag(|\lambda_{1}|, \ldots , |\lambda_{n}|)$.

\begin{proposition}
\label{th:stab1}
Let us assume that there exist $\nu>0$, $\mu\in \RR$ and symmetric positive definite matrices $Q^-$ in $\RR^{m\times m}$ and
$Q^+$ in $\RR^{(n-m)\times (n-m)}$
such that, defining for each $x$ in $[0,1]$, $\mathcal{Q}(x) =  \diag [e ^{2\mu x} Q ^- ,e ^{-2\mu x} Q^+] $,  $\mathcal{Q}(x)\Lambda = \Lambda \mathcal{Q}(x)$, the following matrix inequalities hold
\begin{equation}
\label{eq:lmi1}
-2 \mu \mathcal{Q}(x)\Lambda^+ + F^\top \mathcal{Q}(x)+\mathcal{Q}(x) F \le -2\nu \mathcal{Q}(x)
\end{equation}
\begin{equation}
\label{eq:lmi2}
 \left(\begin{matrix}
 I_{m}& 0_{m,n-m} \\
      G_{+-}& G_{++}
    \end{matrix}
  \right)^\top
  \mathcal{Q}(0) \Lambda\left(\begin{matrix}
 I_{m}& 0_{m,n-m} \\
      G_{+-}& G_{++}
    \end{matrix}
  \right)  
\leq   \left(\begin{matrix}
      G_{--}& G_{-+} \\
      0_{n-m,m}& I_{n-m} \\
    \end{matrix}
  \right)^\top
   \mathcal{Q}(1)\Lambda \left(\begin{matrix}
      G_{--}& G_{-+} \\
      0_{n-m,m}& I_{n-m} \\
    \end{matrix}
  \right).
 \end{equation}
Then there exists $C$ such that (\ref{eq:stab}) holds and the linear hyperbolic system (\ref{eq:lin})-(\ref{eq:bound}) is GES.
\end{proposition}

\ifthenelse{\boolean{arxivbool}}{\begin{proof}
Let us consider the Lyapunov function, for all $y\in L^2((0,1);\RR^n)$,
$$
V(y)=\int_0^1 y(x)^\top \mathcal{Q}(x) y(x)\ dx.
$$
Since $\mathcal{Q}(x)$ and $\Lambda$ commute and are symmetric, then $\mathcal{Q}(x)\Lambda$ is symmetric, $\partial_x \mathcal{Q}(x) \Lambda =- 2\mu \mathcal{Q}(x) \Lambda ^+$, and 
\begin{equation}
\label{eq:der}
y^\top \mathcal{Q}(x) \Lambda\partial_ x y+\partial_ x y \Lambda \mathcal{Q}(x) y  -2\mu  y ^\top \mathcal{Q}(x)\Lambda  ^+ y= 
{\partial_x} (y^\top \mathcal{Q}(x) \Lambda y).
\end{equation}
Then, computing the time-derivative of $V$ along the solutions of (\ref{eq:lin}) yields the following:
\begin{eqnarray*}
\dot V(y)&=& \int_0^1 ( y^\top \mathcal{Q}(x) \partial_t y  +\partial_t y^\top \mathcal{Q}(x)  y)dx
\\
&=& - \int_0 ^1 y ^\top \mathcal{Q}(x) \Lambda \partial _x   y \ dx  - \int_0 ^1  \partial _x y ^\top \Lambda \mathcal{Q}(x) y \ dx+\int_0 ^1  y ^\top (F^\top  \mathcal{Q}(x)+ \mathcal{Q}(x) F)    y \ dx.
\end{eqnarray*}
Then, Equation (\ref{eq:der}) yields:
\begin{eqnarray*}
\dot V (y)&=& -[y^\top \mathcal{Q}(x)\Lambda y ]_0^1 
- \int_0 ^1 2 \mu 
y^\top \mathcal{Q}(x)\Lambda^+ y \ dx 
 +\int_0 ^1  y^\top (F^\top \mathcal{Q}(x)  +\mathcal{Q}(x) F) y  \ dx \\
& = & y^\top(t,0) \mathcal{Q}(0)\Lambda y(t,0) - y^\top(t,1) \mathcal{Q}(1) \Lambda y(t,1) \\
&&  + \int_0 ^1 y^\top (-2 \mu \mathcal{Q}(x)\Lambda^+ + F^\top \mathcal{Q}(x) +\mathcal{Q}(x) F) y  \ dx \\
& = & \left(\begin{matrix}
      y^-(t,0) \\
      y^+(t,1) \\
    \end{matrix}
  \right)^\top
\left( 
 \left(\begin{matrix}
 I_{m}& 0_{m,n-m} \\
      G_{+-}& G_{++}
    \end{matrix}
  \right)^\top
  \mathcal{Q}(0)\Lambda\left(\begin{matrix}
 I_{m}& 0_{m,n-m} \\
      G_{+-}& G_{++}
    \end{matrix}
  \right)  
\right.  \\
&&  
\left.-  \left(\begin{matrix}
      G_{--}& G_{-+} \\
      0_{n-m,m}& I_{n-m} \\
    \end{matrix}
  \right)^\top
   \mathcal{Q}(1)\Lambda \left(\begin{matrix}
      G_{--}& G_{-+} \\
      0_{n-m,m}& I_{n-m} \\
    \end{matrix}
  \right)
 \right) \left(\begin{matrix}
      y^-(t,0) \\
      y^+(t,1) \\
    \end{matrix}
  \right)\\
&&  + \int_0 ^1 y^\top \left(-2 \mu \mathcal{Q}(x)\Lambda^+ + F^\top \mathcal{Q}(x)+\mathcal{Q}(x) F \right)y\ dx 
\end{eqnarray*}
where the last equality is obtained by the boundary conditions (\ref{eq:bound}).
Then, (\ref{eq:lmi1}) and (\ref{eq:lmi2}) imply that $\dot V (y(t,.))\le -2 \nu V(y(t,.))$ which yields, for all $t\in\RR_+ $, 
$V(y(t,.)) \le V(y^0) e^{-2\nu t}$. By remarking that there exist $\alpha>0$, $\beta>0$ (depending on the eigenvalues of $Q^-$, $Q^+$ and on $\mu$) such that 
$
\alpha \|y(t,.)\|_{L^2((0,1);\RR^n)} \le \sqrt{V(y(t,.))} \le \beta \|y(t,.)\|_{L^2((0,1);\RR^n)}$,
we obtain that, for all $t\in\RR_+ $,
$
 \|y(t,.)\|_{L^2((0,1);\RR^n)} \le \frac{\beta}{\alpha} e^{-\nu t} \|y^0\|_{L^2((0,1);\RR^n)}
$.
\end{proof}}{The complete proof of this proposition is detailed in \cite{Prieur:Girard:Witrant:arxiv:2013} and is based on the inequality $\dot V (y(t,.))\le -2 \nu V(y(t,.))$
along the solutions of (\ref{eq:lin}) with the boundary conditions (\ref{eq:bound}), where the Lyapunov function $V$ is defined by $V(y)=\int_0^1 y(x)^\top \mathcal{Q}(x) y(x)\ dx$.
}

If all the diagonal elements of $\Lambda$ are different, the assumption that $\mathcal{Q}(x)\Lambda = \Lambda \mathcal{Q}(x)$ is equivalent to $Q$ being diagonal positive definite\footnote{This equivalence follows from the computation of matrices $\mathcal{Q}(x)\Lambda$ and $\Lambda \mathcal{Q}(x)$, and from a comparison between each of their entries.}. The main contributions of the previous proposition with respect to the result presented in~\cite{DiagneBastinCoron:automatica:11} is double: first, we do not restrict the values of parameter $\mu$ to be positive, this allows us to consider non-contractive boundary conditions (it will be the case for the numerical illustration considered in Example \ref{ex:B}); second, we provide an estimate of the exponential convergence rate\ifthenelse{\boolean{arxivbool}}{ (see Section \ref{sec:comp}
 for computational aspects of this estimate)}{ (see \cite{Prieur:Girard:Witrant:arxiv:2013} for more details)}.
\ifthenelse{\boolean{arxivbool}}{When all the diagonal elements of the matrix $\Lambda$ are positive, then 
Proposition~\ref{th:stab1} can be interpreted in terms of two finite dimensional linear systems that share a common Lyapunov function: one in continuous-time associated to (\ref{eq:lin}) and  one in discrete-time associated to (\ref{eq:bound}). Indeed we have the following result:

\begin{corollary}\label{cor:2.2}
Let us assume that $m=0$ and there exists a diagonal positive definite matrix $M$ in $\RR^{n\times n}$ such that $V:$ $y\in\RR^n \mapsto y^TMy$ is a common Lyapunov function
for the continuous-time and discrete-time linear systems
\begin{equation}
\label{eq:cont}
\dot y(t) = \left(\Lambda^{-1} F -\mu I\right) y(t), \quad t\in \RR_+,
\end{equation}
\begin{equation}
\label{eq:disc}
y(t+1) = \left(e^{\mu}G \right) y(t), \quad t\in \NN.
\end{equation}
Then, the linear hyperbolic system (\ref{eq:lin})-(\ref{eq:bound}) is GES.
\end{corollary}

\begin{proof} Remark first that $\Lambda^+=\Lambda$ since it is assumed that $m=0$. Let $Q=M\Lambda^{-1}$, then $Q$ is diagonal positive definite. 
By writing the Lyapunov equation of the continuous-time system (\ref{eq:cont}), we obtain
$
\left(\Lambda^{-1} F -\mu I\right)^\top M + M \left(\Lambda^{-1} F -\mu I\right) <0
$
which can be rewritten as
$
\left( F -\mu \Lambda\right)^\top Q + Q \left( F -\mu \Lambda \right)=-2 \mu Q\Lambda + F^\top Q +Q F  <0.
$
This implies existence of $\nu>0$ such that (\ref{eq:lmi1}) holds.
Also the Lyapunov equation of the discrete-time system (\ref{eq:disc}) gives
$
e^{\mu}G^\top M e^{\mu} G \le M.
$
This can be rewritten as 
$
G^\top Q\Lambda G \le e^{-2\mu}Q\Lambda
$
which is equivalent to (\ref{eq:lmi2}), since $m=0$ and $\mathcal{Q}(x)=e^{-2\mu x} Q$. Thus the assumptions of Proposition~\ref{th:stab1} hold and this concludes the proof of Corollary \ref{cor:2.2}.
\end{proof}

Let us remark that increasing $\mu$  improves the stability of (\ref{eq:cont}) and degrades that of (\ref{eq:disc})
while decreasing $\mu$ will have the reverse effect. Another interpretation of the previous result is 
that the linear hyperbolic system (\ref{eq:lin})-(\ref{eq:bound}) is GES if there is a balance between the  expansion (respectively contraction) rate of the
continuous-time linear system
$
\dot y(t) = \Lambda^{-1} Fy(t)
$ 
and the contraction (respectively expansion) rate of discrete-time linear system $y(t+1) = Gy(t)$.}{}

\section{Switched linear hyperbolic systems}\label{sec:3}

We now consider the case of switched linear hyperbolic partial differential equation of the form (see \cite{AminHanteBayen:IEEETAC:2011})
\begin{equation}
\label{eq:slin}
\partial_t w(t,x)+L_{\sigma(t)}  \partial_ x
  w(t,x) = A_{\sigma(t)}w(t,x) \;, \quad x \in [0,1],\; t\in \RR_+
\end{equation}
where $w:\RR_+\times [0,1] \rightarrow \RR^n$, $\sigma:\RR_+ \rightarrow I$, $I$ is a finite set (of {\em modes}), $A_i$ and $L_i$ are matrices in $\RR^{n\times n}$, for $i\in I$. 
The partial differential equation associated with each mode is hyperbolic, meaning that for all $i\in I$, there exists an invertible matrix $S_i$ in $\RR^{n\times n}$ such that $L_i=S_i^{-1} \Lambda_i S_i$ where $\Lambda_i$ is a diagonal matrix in $\RR^{n\times n}$ satisfying
$\Lambda_i=\diag(\lambda_{i,1}, \ldots , \lambda_{i,n})$, with
$\lambda_{i,k} <0$ for $k\in\{1,\dots,m_i\}$ and $\lambda_{i,k}>0$ for $k\in \{m_i+1,\dots,n\}$. 
The matrices $S_i$ can be written as
\begin{equation}
\label{cp:8:11}
S_i=\left(
\begin{matrix} S_i^{-\top} &S_i^{+\top} \end{matrix}
\right)^\top 
\end{equation}
 where $S_i^-$ and $S_i^+$ are matrices in $\RR^{m_i\times n}$ and 
$\RR^{(n-m_i)\times n}$.
We define the matrices $F_i=S_i A_i S_i^{-1}$ and
$\Lambda_i^+=\diag(|\lambda_{i,1}|, \ldots , |\lambda_{i,n}|)$
for $i\in I$.
The boundary conditions are given by
\begin{equation}
\label{eq:sbound}
B^0_{\sigma(t)} w(t,0)+B^1_{\sigma(t)} w(t,1)=0, \quad t\ge 0
\end{equation}
where, for all $i\in I$, $B^0_i=G^0_i S_i$ and $B^1_i=G^1_i S_i$, $G^0_i$ and $G^1_i$ being matrices in $\RR^{n\times n}$ that satisfy
$$
G^0_i=
\left(
\begin{smallmatrix}
 -G_{i--} & 0_{m_i,n-m_i} \\
- G_{i+-}  & I_{n-m_i}
\end{smallmatrix}
\right),\;
G^1_i=
\left(
\begin{smallmatrix}
I_{m_i} & -G_{i-+} \\
0_{m_i,n-m_i} & - G_{i++} 
\end{smallmatrix}
\right).
$$
For $i\in I$, let us define the matrices in $\RR^{n\times n}$, $G_i=\left( \begin{smallmatrix}G_{i--} & G_{i-+}\\ G_{i+-} & G_{i++} \end{smallmatrix} \right)$.
We shall consider an initial condition given by 
\begin{equation}
\label{eq:sinit}
w(0,x)=w^0(x), \quad  x \in (0,1)
\end{equation}
where $w^0\in L^2((0,1);\RR^n)$.

A {\em switching signal} is a piecewise constant function $\sigma:\RR_+\rightarrow I$, right-continuous, and with a finite number of discontinuities on every bounded interval of $\RR_+$. This allows us to avoid the {\em Zeno behavior}, as described in \cite{liberzon:book:2003}.
The set of switching signals is denoted by $\mathcal{S}(\RR_+,I)$. The discontinuities of $\sigma$ are called {\em switching times}. 
The number of discontinuities of $\sigma$ on the interval $(\tau,t]$ is denoted by $N_\sigma(\tau,t)$. Following~\cite{HespanhaMorse:CDC:99}, 
for $\tau_D>0$, $N_0\in \NN$,
we denote by $\mathcal{S}_{\tau_D,N_0}(\RR_+,I)$ the set of switching signals verifying, for all $\tau<t$,
$
N_\sigma(\tau,t) \le N_0 +\frac{t-\tau}{\tau_D}.
$
The constant $\tau_D$ is called the {\em average dwell time} and $N_0$ the {\em chatter bound}. 

We first provide an existence and uniqueness result for the solutions of (\ref{eq:slin})-(\ref{eq:sinit}):

\begin{proposition}\label{prop:existence} For all $\sigma \in \mathcal{S}(\RR_+,I)$, $w^0\in L^2((0,1);\RR^n)$, there exists a unique (weak) solution $w\in C^0(\RR_+;L^2((0,1);\RR^n))$ to the initial value problem (\ref{eq:slin})-(\ref{eq:sinit}). 
\end{proposition}

{
\begin{proof} We build iteratively the solution between successive switching times. Let $(t_k)_{k\in K}$ denote the increasing switching times of $\sigma$, with $t_0=0$ and $K$ be a (finite or infinite) subset of $\NN$.
Let us assume that we have been able to build a unique (weak) solution $w\in C^0([0,t_k];L^2((0,1);\RR^n))$ for some $k\ge 0$. Then, let $i_k$ be the value of 
$\sigma(t)$ for $t\in [t_k,t_{k+1})$. Let us introduce the following notation, for all $k$ in $K$ and for all $x$ in $[0,1]$,
\begin{equation}
\label{eq:change-var}
y_k(t,x)= S_{i_k} w(t,x), \quad t\in [t_k,t_{k+1}].
\end{equation}
Note that closed time intervals are used on both sides due to technical reasons in this proof.
Then, (\ref{eq:slin}) gives that, for all $k$ in $K$, $y_k$ satisfies the following partial differential equation
\begin{equation}
\label{eq:lin_k}
\partial_t y_k(t,x)+\Lambda_{i_k}  \partial_ x
  y_k(t,x) = F_{i_k}y_k(t,x), \quad  x \in[0,1],\; t\in [t_k,t_{k+1}]\ .
\end{equation}
Also, we use the notations $y_k=\left(\begin{smallmatrix} y_k^-\\y_k^+\end{smallmatrix}\right)$, 
where $y_k^-:\RR_+\times [0,1] \rightarrow\RR^{m_{i_k}}$ and $y_k^+:\RR_+\times [0,1] \rightarrow\RR^{n-m_{i_k}}$.
The boundary conditions (\ref{eq:sbound}) give, for all $k$ in $K$,
 \begin{equation}\label{eq:bound_k}
  \left(\begin{smallmatrix}
      y_k^-(t,1) \\
      y_k^+(t,0) \\
    \end{smallmatrix}
  \right)
  =G_{i_k}
  \left(\begin{smallmatrix}
      y_k^-(t,0) \\
      y_k^+ (t,1) \\
    \end{smallmatrix}
  \right), \quad   t\in [t_k,t_{k+1}]\ .
\end{equation}
The initial condition ensuring the continuity of $w$ at time $t_k$ is the following:
\begin{equation}
\label{eq:init_k}
y_{k}(t_k,x)=S_{i_k} w(t_k,x), \quad  x \in (0,1)\ .
\end{equation}
It follows from~\cite{DiagneBastinCoron:automatica:11} that, for all $k$ in $K$, there exists a unique (weak) solution 
\\
$y_{k} \in C^0([t_k,t_{k+1}];L^2((0,1);\RR^n))$ to the initial value problem (\ref{eq:lin_k})-(\ref{eq:init_k}).
Then, we can extend the (weak) solution to the initial value problem (\ref{eq:slin})-(\ref{eq:sinit}), from the initial time $t_k$, up to the switching time $t_{k+1}$;
(\ref{eq:init_k}) ensures that $w\in C^0([0,t_{k+1}];L^2((0,1);\RR^n))$, and the uniqueness of $y_k$ ensures that $w$ is the unique solution.
Finally, since there are only a finite number of switching times on every bounded intervals of $\RR_+$, the solution can be defined for all times,
resulting on a unique solution $w\in C^0(\RR_+;L^2((0,1);\RR^n))$.
\end{proof}}{}

\section{Stability of switched linear hyperbolic systems}\label{sec:4}

Let $\mathcal S\subseteq \mathcal{S}(\RR_+,I)$. The switched linear hyperbolic system (\ref{eq:slin})-(\ref{eq:sbound}) is said to be {\em globally uniformly exponentially stable} (GUES) 
with respect to the set of switching signals $\mathcal S$ if there exist $\nu >0$ and
$C>0$ such that, for every $w_0\in L^2((0,1);\RR^n)$, for every $\sigma \in \mathcal S$,
the solution to the initial value problem (\ref{eq:slin})-(\ref{eq:sinit}) satisfies
$
\|w(t,.)\|_{L^2((0,1);\RR^n)} \le C e^{-\nu t} \|w^0\|_{L^2((0,1);\RR^n)} $, $\forall t\in \RR_+ .
$
In this section, we provide sufficient conditions for the stability of switched linear hyperbolic systems.

\subsection{Mode independent sign structure of characteristics}\label{sec:4.1}

Assume first that the number of negative and positive characteristics of the linear partial differential equations associated with each mode is constant, that is 
for all $i\in I$, $m_i=m$.

We provide a first result giving sufficient conditions such that stability holds for all switching signals. \startchris\ifthenelse{\boolean{arxivbool}}{The proof is based on a common Lyapunov function equivalent to the $L^2$ norm}{See \cite{Prieur:Girard:Witrant:arxiv:2013} for a proof where a common Lyapunov function equivalent to the $L^2$-norm is used}. An alternative proof can be obtained by checking some semigroup properties and by using \cite{HanteSigalotti:SIAM:11} (where the equivalence is shown between the existence of a common Lyapunov function commensurable
with the squared norm and the global uniform exponential stability).\stopchris

\begin{theorem}
\label{th:sstab1}
Let us assume that, for all $i\in I$, $m_i=m$ and that there exist $\nu>0$, $\mu\in \RR$ and diagonal positive definite matrices $Q_i$ in $\RR^{n\times n}$, $i\in I$ such that the following matrix inequalities hold, for all $i\in I$ and for all 
$x$ in $[0,1]$,
\begin{equation}
\label{eq:s1lmi1}
-2 \mu \mathcal{Q}_i(x)\Lambda_i^+ + F_i^\top \mathcal{Q}_i(x)+\mathcal{Q}_i(x) F_i \le -2\nu \mathcal{Q}_i(x),
\end{equation}
\begin{equation}
\label{eq:s1lmi2}
 \left(\begin{matrix}
 I_{m}& 0_{m,n-m} \\
      G_{i+-}& G_{i++}
    \end{matrix}
  \right)^\top
  \mathcal{Q}_i(0) \Lambda_i\left(\begin{matrix}
 I_{m}& 0_{m,n-m} \\
      G_{i+-}& G_{i++}
    \end{matrix}
  \right)  
\leq   \left(\begin{matrix}
      G_{i--}& G_{i-+} \\
      0_{n-m,m}& I_{n-m} \\
    \end{matrix}
  \right)^\top
   \mathcal{Q}_i(1)\Lambda_i \left(\begin{matrix}
      G_{i--}& G_{i-+} \\
      0_{n-m,m}& I_{n-m} \\
    \end{matrix}
  \right),
\end{equation}
where $\mathcal{Q}_i(x) =  \diag [e ^{2\mu x} Q_i  ^- ,e ^{-2\mu x} Q_i ^ +] $,
$Q_i=\left(
\begin{smallmatrix}
Q_i^- & 0 \\
0 & Q_i^+
\end{smallmatrix}
 \right)$, $Q_i^-$ and $Q_i^+$ are diagonal positive matrices in $\RR^{m_i\times m_i}$ and 
$\RR^{(n-m_i)\times (n-m_i)}$, together with the following matrix equalities, for all $i,j\in I$,
\begin{equation}
\label{eq:s1lmi3}
\begin{array}{c}
(S_i^+)^\top Q_i^+ S_i^+=(S_j^+)^\top Q_j^+ S_j^+ ,
\quad 
(S_i^-)^\top Q_i^- S_i^-=(S_j^-)^\top Q_j^- S_j^- .
\end{array}
\end{equation}
Then, the switched linear hyperbolic system (\ref{eq:slin})-(\ref{eq:sbound}) is GUES
with respect to the set of switching signals $\mathcal{S}(\RR_+,I)$.
\end{theorem}

\ifthenelse{\boolean{arxivbool}}{\begin{proof}  
\startchris Given the diagonal matrices $Q_i$ satisfying the assumptions of Theorem \ref{th:sstab1}, \stopchris let $M^-=(S_i^-)^\top Q_i^- S_i^-$ and $M^+=(S_i^+)^\top Q_i^+ S_i^+$, by (\ref{eq:s1lmi3}),
these matrices do not depend on the index $i\in I$.
The proof is based on the use of a common Lyapunov function given by,  for all $w$ in $L^2((0,1);\RR^n)$,
$$
V(w)=\int_0^1 w(x)^\top \mathcal{M}(x) w(x)\ dx,
$$
where $\mathcal{M}(x) = e^{2\mu x}M^-+ e^{-2\mu x}M^+$. 
Let $(t_k)_{k\in K}$ denote the increasing switching times of $\sigma$, with $t_0=0$ and $K$
a (finite or infinite) subset of $\NN$.
For $k\in K$, let $i_k$ be the value of 
$\sigma(t)$ for $t\in [t_k,t_{k+1})$, and let $y_k$ be given by (\ref{eq:change-var}). It thus
satisfies the boundary conditions (\ref{eq:bound_k}).
Let us remark that, due to (\ref{cp:8:11}) and (\ref{eq:init_k}), for $t\in [t_k,t_{k+1})$, $V$ can be written as:
$$
V(w(t,.))=\int_0^1 y_k(t,x)^\top \mathcal{Q}_{i_k}(x) y_k(t,x) \ dx, 
$$
where $\mathcal{Q}_{i_k}(x) =\diag[e^{2\mu x}Q_{i_k}^-,e^{-2\mu x}Q_{i_k}^+ ]$. \startchris Note that $\mathcal{Q}_{i_k}(x)$ commute with $\Lambda$ since these matrices are diagonal. \stopchris
Using (\ref{eq:s1lmi1}) and (\ref{eq:s1lmi2}) and following the proof of Proposition~\ref{th:stab1}, we obtain that, along the solutions of  (\ref{eq:slin})-(\ref{eq:sbound}), it holds, for all $k$ in $K$,
\begin{equation}
\label{eq:lyapt}
\forall t\in [t_k,t_{k+1}),\; V(w(t,.)) \le V(w(t_k,.))e^{-2\nu(t-t_k)}.
\end{equation}
Moreover, $V(w(t,.))$ is continuous at the switching time $t_{k+1}$, thus
\begin{equation}
\label{eq:lyapk}
\forall k\in K,\; V(w(t_{k+1},.)) \le V(w(t_k,.))e^{-2\nu(t_{k+1}-t_k)}.
\end{equation}
Equations (\ref{eq:lyapt}) and (\ref{eq:lyapk}) allow us to prove that, for all $t\in \RR_+$, 
$V(w(t,.)) \le V(w^0) e^{-2\nu t}$.
By noting that there exist $\alpha>0$, $\beta>0$ such that
$
\alpha \|w(t,.)\|_{L^2((0,1);\RR^n)} \le \sqrt{V(w(t,.))} \le \beta \|w(t,.)\|_{L^2((0,1);\RR^n)}$,
we obtain that, for all $t\in \RR_+$,
$
 \|w(t,.)\|_{L^2((0,1);\RR^n)} \le \frac{\beta}{\alpha} e^{-\nu t} \|w^0\|_{L^2((0,1);\RR^n)}
$. This concludes the proof of Theorem \ref{th:sstab1}.
\end{proof}}{}

The numerical computation of the unknown variables, satisfying the sufficient conditions of Theorem \ref{th:sstab1}, is explained in \ifthenelse{\boolean{arxivbool}}{Section \ref{sec:comp} below}{\cite{Prieur:Girard:Witrant:arxiv:2013}}.

For systems that do not satisfy the assumptions of the previous theorem, but whose dynamics in each mode satisfy independently 
the assumptions of Proposition~\ref{th:stab1} (i.e. the dynamics in each mode is stable), it is possible to show that the system
is stable provided that the switching is slow enough:

\begin{theorem}
\label{th:sstab2}
Let us assume that, for all $i\in I$, $m_i=m$ and that there exist $\nu>0$, $\gamma\ge 1$, $\mu_i\in \RR$, diagonal positive definite matrices $Q_i$ in $\RR^{n\times n}$, such that the following matrix inequalities hold, for all 
$x$ in $[0,1]$,
\begin{equation}
\label{eq:s2lmi1}
-2 \mu_i \mathcal{Q}_i(x)\Lambda_i^+ + F_i^\top \mathcal{Q}_i(x)+\mathcal{Q}_i(x) F_i \le -2\nu \mathcal{Q}_i(x),
\end{equation}
\begin{equation}
\label{eq:s2lmi2}
 \left(\begin{matrix}
 I_{m}& 0_{m,n-m} \\
      G_{i+-}& G_{i++}
    \end{matrix}
  \right)^\top
  \mathcal{Q}_i(0) \Lambda_i\left(\begin{matrix}
 I_{m}& 0_{m,n-m} \\
      G_{i+-}& G_{i++}
    \end{matrix}
  \right)  
\leq   \left(\begin{matrix}
      G_{i--}& G_{i-+} \\
      0_{n-m,m}& I_{n-m} \\
    \end{matrix}
  \right)^\top
   \mathcal{Q}_i(1)\Lambda_i \left(\begin{matrix}
      G_{i--}& G_{i-+} \\
      0_{n-m,m}& I_{n-m} \\
    \end{matrix}
  \right),
\end{equation}
where $\mathcal{Q}_i(x) =  \diag [e ^{2\mu_i x} Q_i  ^- ,e ^{-2\mu_i x} Q_i ^ +] $,
$Q_i=\left(
\begin{smallmatrix}
Q_i^- & 0 \\
0 & Q_i^+
\end{smallmatrix}
 \right)$, $Q_i^-$ and $Q_i^+$ are diagonal positive matrices in $\RR^{m_i\times m_i}$ and 
$\RR^{(n-m_i)\times (n-m_i)}$, together with the following matrix inequalities, for all $i,j\in I$,
\begin{equation}
\label{eq:s2lmi3}
(S_i^+)^\top Q_i^+ S_i^+\le \gamma (S_j^+)^\top Q_j^+ S_j^+ ,
\end{equation}
\begin{equation}
\label{eq:s2lmi4}
(S_i^-)^\top Q_i^- S_i^-\le \gamma (S_j^-)^\top Q_j^- S_j^- .
\end{equation}

Let $\Delta_\mu=\max(\mu_1,\dots,\mu_n)-\min(\mu_1,\dots,\mu_n)$,
then, for all $N_0\in \NN$, for all $\tau_D> \frac{\ln(\gamma)}{2\nu}+\frac{\Delta_\mu}{\nu}$,
the switched linear hyperbolic system (\ref{eq:slin})-(\ref{eq:sbound}) is GUES
with respect to the set of switching signals $\mathcal{S}_{\tau_D,N_0}(\RR_+,I)$.
\end{theorem}

\begin{proof} 
Let $(t_k)_{k\in K}$ denote the increasing switching times of $\sigma$, with $t_0=0$ and $K$ is a (finite or infinite) subset of $\NN$.
For $k\in K$, let $i_k$ be the value of 
$\sigma(t)$ for $t\in [t_k,t_{k+1})$, and let $y_k$ be given by (\ref{eq:change-var}). It satisfies the boundary conditions (\ref{eq:bound_k}). \stopchris
\startchris Given the diagonal matrices $Q_i$ satisfying the assumptions of Theorem \ref{th:sstab2}, \stopchris
let $M_i^-=(S_i^-)^\top Q_i^- S_i^-$ and $M_i^+=(S_i^+)^\top Q_i^+ S_i^+$. 
\stopchris 
The proof is based on the use of multiple Lyapunov functions. More precisely, denoting $\mathcal{M}_{i_k}(x) = e^{2\mu_{i_k} x}M_{i_k}^- + e^{-2\mu_{i_k} x}M_{i_k}^+$, let us define, for all $w$ in $C^0([0,\infty); L^2((0,1);\RR^n))$, for all $t$ in $\RR_+$,
\begin{equation}\label{21bis}
V(w(t,.))=\int_0^1 w(t,x)^\top \mathcal{M}_{i_k}(x) w(t,x)\ dx, \mbox{ if } t\in[t_k,t_{k+1}) 
\end{equation}
which may be rewritten as
$
V(w(t,.))=\int_0^1 y_k(t,x)^\top \mathcal{Q}_ {i_k}(x) y_k(t,x)\ dx$, if $t\in[t_k,t_{k+1}) 
$.

\startchris Note that $\mathcal{Q}_{i_k}(x)$ commute with $\Lambda_{i_k}$ since these matrices are diagonal. \stopchris
Using (\ref{eq:s2lmi1}) and (\ref{eq:s2lmi2}), 
and following the proof of Proposition \ref{th:sstab1}, 
we get that, along the solutions of (\ref{eq:slin})-(\ref{eq:sbound}), 
\begin{equation}
\label{eq:lyapt:2}
\forall k \in K, \forall t\in [t_k,t_{k+1}),\; V(w(t,.)) \le V(w(t_k,.))e^{-2\nu(t-t_k)}.
\end{equation}
The function $V$ may be not continuous at the switching times any more. Nevertheless, by (\ref{eq:s2lmi3}) and (\ref{eq:s2lmi4}), we have that, for all $k$ in $K$,
\begin{eqnarray*}
V(w(t_{k+1},.))& = & \int_0^1 \left(w(t_{k+1},x)^\top M_{i_{k+1}}^- w(t_{k+1},x)e^{2\mu_{i_{k+1}} x}
 +
w(t_{k+1},x)^\top M_{i_{k+1}}^+ w(t_{k+1},x)e^{-2\mu_{i_{k+1}} x} \right)   dx\\
 & \le  &  \gamma \int_0^1 \left(w(t_{k+1},x)^\top M_{i_{k}}^- w(t_{k+1},x)e^{2\mu_{i_{k+1}} x}
 +
w(t_{k+1},x)^\top M_{i_{k}}^+ w(t_{k+1},x)e^{-2\mu_{i_{k+1}} x} \right)   dx\\
& \le & \gamma e^{2\Delta_\mu} \int_0^1 \left(w(t_{k+1},x)^\top M_{i_{k}}^- w(t_{k+1},x)e^{2\mu_{i_{k}} x} +
w(t_{k+1},x)^\top M_{i_{k}}^+ w(t_{k+1},x)e^{-2\mu_{i_{k}} x} \right)  dx \\
& \le & \gamma e^{2\Delta_\mu} \lim_{t\rightarrow t_{k+1}^-} V(w(t,.))
\end{eqnarray*}
where the continuity of $w$ is used in the last inequality (it follows from Proposition \ref{prop:existence}). Then, it follows from (\ref{eq:lyapt:2}) that, for all $k$ in $K$, $V(w(t_{k+1},.)) \le \gamma e^{2\Delta_\mu} V(w(t_k,.))e^{-2\nu(t_{k+1}-t_k)}$,
and it allows us to 
prove recursively that, for all $ t\in\RR_+$,
$
V(w(t,.)) \le \left(\gamma e^{2\Delta_\mu}\right)^{N_\sigma(0,t)} V(w^0)e^{-2\nu t} \le
 \left(\gamma e^{2\Delta_\mu}\right)^{(N_0+\frac{t}{\tau_D})} V(w^0)e^{-2\nu t}.
$
Let $\bar\nu=\nu-\frac{\Delta_\mu}{\tau_D}-\frac{\ln(\gamma)}{2\tau_D}$, the assumption on the average dwell time gives that $\bar \nu>0$ and the previous inequality yields
$
\forall t\in\RR_+,\; V(w(t,.)) \le 
 \left(\gamma e^{2\Delta_\mu}\right)^{N_0} V(w^0)e^{-2\bar \nu t}
$
which allows us to conclude that the switched linear hyperbolic system is GUES
with respect to the set of switching signals $\mathcal{S}_{\tau_D,N_0}(\RR_+,I)$. This concludes the proof of Theorem \ref{th:sstab2}.
\end{proof}

\begin{remark}
Setting $\gamma=1$ and $\mu_i=\mu$ for all $i\in I$, we recover the assumptions of Theorem~\ref{th:sstab1}.
In that case we have $\Delta_\mu=0$: there is no positive lower bound imposed on the average dwell time, which is consistent with Theorem~\ref{th:sstab1}. 
\end{remark}
\startchris
\begin{remark}
Note that the existence of $\gamma\geq 1$ such that (\ref{eq:s2lmi3}) holds is equivalent to $\ker (S_i^+)=\ker (S_j^+)$, for all 
$i,j\in I$. Therefore the existence of $\gamma\geq 1$ such that (\ref{eq:s2lmi3}) and (\ref{eq:s2lmi4}) are satisfied is equivalent to $\ker (S_i^+)=\ker (S_j^+)$ and $\ker (S_i^-)=\ker (S_j^-)$, for all 
$i,j\in I$ (and also, by recalling $L_i=S_i^{-1} \Lambda_i S_i$, the subspace associated with all positive (resp. negative) eigenvalues of $L_i$ does not depend on $i$). If this condition does not hold, stability can still be analyzed using other stability results presented in the following section. 
\end{remark}
\stopchris
\subsection{Mode dependent sign structure of characteristics}
\label{sec:stabedep}

We now relax the assumption on the number of negative and positive characteristics. 
As in the previous section, we provide a first result giving sufficient conditions such that stability holds for all switching signals\ifthenelse{\boolean{arxivbool}}{:}{ (see \cite{Prieur:Girard:Witrant:arxiv:2013} for the proof):}
\begin{theorem}
\label{th:sstab3}
Let us assume that there exist $\nu>0$ and diagonal positive definite matrices $Q_i$ in $\startchris \RR^{n\times n}\stopchris$, $i\in I$ 
such that, for all $i\in I$, 
\begin{equation}
\label{eq:s3lmi1}
F_i^\top Q_i +Q_i F_i \le -2\nu Q_i,
\end{equation}
\begin{equation}
\label{eq:s3lmi2}
G_i^TQ_i\Lambda_i^+G_i \le 
Q_i\Lambda^+_i,
\end{equation}
and such that, for all $i,j\in I$,
\begin{equation}
\label{eq:s3lmi3}
S_i^\top Q_iS_i=S_j^\top Q_jS_j .
\end{equation}
Then, the switched linear hyperbolic system (\ref{eq:slin})-(\ref{eq:sbound}) is GUES
with respect to the set of switching signals $\mathcal{S}(\RR_+,I)$.
\end{theorem}

\ifthenelse{\boolean{arxivbool}}{\begin{proof} We use the same notations as in Theorem~\ref{th:sstab1}.  We consider the candidate Lyapunov function, for all $w$ in $C^0([0,\infty); L^2((0,1);\RR^n))$, for all $t\in\RR_+$,
$$
V(w(t,.))=\int_0^1 y_k(t,x)^\top Q_{i_k} y_k(t,x) dx, \quad \text{if } t\in [t_k,t_{k+1}),
$$
where we used the change of variable (\ref{eq:change-var}).
Using (\ref{eq:s3lmi1}) and (\ref{eq:s3lmi2}), and following the proof of Proposition~\ref{th:stab1}, we obtain that, along the solutions of  (\ref{eq:slin})-(\ref{eq:sbound}), it holds, for all $k$ in $K$, and for all $
 t\in [t_k,t_{k+1})$, $V(w(t,.)) \le V(w(t_k,.))e^{-2\nu(t-t_k)}.
$
Recalling (\ref{eq:change-var}), we have
$y_k(t_{k+1},.)=S_{i_k}w(t_{k+1},.)$ and $y_{k+1}(t_{k+1},.)=S_{i_{k+1}}w(t_{k+1},.)$, which gives
$y_{k+1}(t_{k+1},.)=S_{i_{k+1}}S^{-1}_{i_k}y_k(t_{k+1},.)$.
Hence, Equation (\ref{eq:s3lmi3}) yields, for all $k$ in $K$,
\begin{eqnarray*}
V(w(t_{k+1},.))& = & \int_0^1 y_{k+1}(t_{k+1},x)^\top Q_{i_{k+1}} y_{k+1}(t_{k+1},x) dx \\
 & = & \int_0^1 y_k(t_{k+1},x)^\top S^{-\top}_{i_k} S_{i_{k+1}}^\top Q_{i_{k+1}}S_{i_{k+1}}S^{-1}_{i_k}y_k(t_{k+1},x) dx \\
& = & \int_0^1 y_k(t_{k+1},x)^\top Q_{i_k} y_k(t_{k+1},x)  dx = \lim_{t\rightarrow t_{k+1}^-} V(w(t,.)).
\end{eqnarray*}
Thus, $V$ is continuous at the switching time $t_{k+1}$. The end of the proof is similar to that of Theorem~\ref{th:sstab1}.
\end{proof}}{}

The assumptions of the previous theorem are quite strong. To assure the asymptotic stability for switching signals with a sufficiently large dwell time, weaker assumptions are needed. More precisely, considering the assumptions of Theorem \ref{th:sstab2}, the last main result of this paper can be stated:
\begin{theorem}
\label{th:sstab4}
Let us assume that there exist $\nu>0$, $\gamma\geq 1$, $\mu_i\in \RR$, and diagonal positive definite matrices $Q_i$ in \startchris $\RR^{n\times n}$\stopchris, $i\in I$ 
such that the matrix inequalities (\ref{eq:s2lmi1}), (\ref{eq:s2lmi2}) hold \startchris (where the same notation for $\mathcal{Q}_i(x)$ is used) \stopchris together with the following matrix inequalities, for all $i,j\in I$,
\begin{equation}
\label{eq:s4lmi}
S_i^\top Q_i S_i \le \gamma S_j^\top Q_j S_j.
\end{equation}
Let $\bar\Delta_\mu=2|\mu_i|$ if $I$ is a singleton and $\bar\Delta_\mu=2\max_{i\neq j \in I}(|\mu_i|+|\mu_j|)$ else.
Then, for all $N_0\in \NN$, for all $\tau_D> \frac{\ln(\gamma)}{2\nu}+\frac{\bar\Delta_\mu}{\nu}$,
the switched linear hyperbolic system (\ref{eq:slin})-(\ref{eq:sbound}) is GUES
with respect to the set of switching signals $\mathcal{S}_{\tau_D,N_0}(\RR_+,I)$.
\end{theorem}

\begin{proof}
We use the same notations as in Theorem \ref{th:sstab2}, and we consider the candidate Lyapunov function (\ref{21bis}).
\startchris Using (\ref{eq:s2lmi1}) and (\ref{eq:s2lmi2}), Equation (\ref{eq:lyapt:2}) \stopchris still holds along the solutions of (\ref{eq:slin})-(\ref{eq:sbound}). Moreover, for all $k$ in $K$,
\begin{eqnarray*}
V(w(t_{k+1},.))
          & \leq & e^{2|\mu_{i_{k+1}}|} \int_0^1 y_{k+1}(t_{k+1},x)^\top Q_{i_{k+1}} y_{k+1}(t_{k+1},x) dx \\
          & \leq & \gamma e^{2|\mu_{i_{k+1}}|} \int_0^1 y_{k+1}(t_{k+1},x)^\top Q_{i_{k}} y_{k+1}(t_{k+1},x) dx \\
          & \le & \gamma e^{2|\mu_{i_{k+1}}|+2|\mu_{i_{k}}|} \int_0^1 y_k(t_{k+1},x)^\top \mathcal{Q}_{i_k}(x) y_k(t_{k+1},x) dx \\
          & \le &\gamma e^{2\bar\Delta_\mu} \lim_{t\rightarrow t_{k+1}^-} V(w(t,.)).
\end{eqnarray*}
The end of the proof follows the same lines as that of Theorem~\ref{th:sstab2}.
\end{proof}

\startchris
Let us note that 
Theorem \ref{th:sstab3} can be deduced from Theorem \ref{th:sstab4} by selecting $\gamma=1$ and $\mu_i=0$ for all $i\in I$. 
\stopchris

\ifthenelse{\boolean{arxivbool}}{
\section{Computational aspects}
\label{sec:comp}

The sufficient stability conditions of the results presented in this paper may be solved using classical numerical tools. More precisely, let us remark that the matrix inequalities (\ref{eq:s1lmi1}) and (\ref{eq:s1lmi2}) in the statement of Theorem \ref{th:sstab1} are linear in $Q_i$ but nonlinear in $\mu$, some numerical methods may be used:
\begin{itemize}
\item by particularizing to the case $\mu=0$;
\item or when the source terms in (\ref{eq:slin})-(\ref{eq:sbound}) are diagonal;
\item or when the matrices $\Lambda_i$ are either all positive definite or all negative definite.
\end{itemize}

Let us consider these three cases in the next three sections.

\subsection{Particularizing to the case $\mu=0$}

By letting $\mu=0$ in the conditions (\ref{eq:s1lmi1}) and (\ref{eq:s1lmi2}), the matrix inequalities in Theorem \ref{th:sstab1} do not depend on the $x$-variable. Therefore the following matrix inequalities 
\begin{equation}
\label{eq:s1lmi1:lin}
 F_i^\top Q_i +Q_i F_i \le -2\nu I_n\ , \ Q_i \le I_n.
\end{equation}
\begin{equation}
\label{eq:s1lmi2:lin}
 \left(\begin{matrix}
 I_{m}& 0_{m,n-m} \\
      G_{i+-}& G_{i++}
    \end{matrix}
  \right)^\top
  Q _i \Lambda_i\left(\begin{matrix}
 I_{m}& 0_{m,n-m} \\
      G_{i+-}& G_{i++}
    \end{matrix}
  \right)  
\leq   \left(\begin{matrix}
      G_{i--}& G_{i-+} \\
      0_{n-m,m}& I_{n-m} \\
    \end{matrix}
  \right)^\top
   Q_i \Lambda_i \left(\begin{matrix}
      G_{i--}& G_{i-+} \\
      0_{n-m,m}& I_{n-m} \\
    \end{matrix}
  \right)
\end{equation}
imply (\ref{eq:s1lmi1}) and (\ref{eq:s1lmi2}).

The previous conditions (\ref{eq:s1lmi1:lin}) and (\ref{eq:s1lmi2:lin}) are linear in the unknown variables $\nu$ and $Q_i$ and can therefore be solved using semi-definite programming (see e.g., \cite{Lofberg_CACSDConf_04} with \cite{Sturm_OptMethSoft_99}).
It is thus obtained the following result as a corollary of Theorem \ref{th:sstab1}:
\begin{corollary}\label{corollary:1}
Let us assume that, for all $i\in I$, $m_i=m$, and there exist $\nu>0$, and diagonal positive definite matrices $Q_i$ in $\RR^{n\times n}$ satisfying the matrix inequalities (\ref{eq:s1lmi3}), (\ref{eq:s1lmi1:lin}) and (\ref{eq:s1lmi2:lin}).
Then, the switched linear hyperbolic system (\ref{eq:slin})-(\ref{eq:sbound}) is GUES
with respect to the set of switching signals $\mathcal{S}(\RR_+,I)$.
\end{corollary}

Moreover the proof of Theorem \ref{th:sstab1} implies that  the function given by,  for all $w$ in $L^2((0,1);\RR^n)$,
$$
V(w)=\int_0^1 w(x)^\top M w(x)\ dx,
$$
where $M = (S_i^-)^\top Q_i^- S_i^- + (S_i^+)^\top Q_i^+ S_i^+$
is a Lyapunov function of (\ref{eq:slin})-(\ref{eq:sbound}), and, since the estimation
$$
 \|w(t,.)\|_{L^2((0,1);\RR^n)} \le C e^{-\nu t} \|w^0\|_{L^2((0,1);\RR^n)}
$$
holds along the solutions $w$ of (\ref{eq:slin})-(\ref{eq:sbound}), for a suitable value $C>0$ (which does not depend on the solution), it implies that the value of $\nu$ is an estimation of the speed of the exponential stability.
Semi-definite programming (as in \cite{Sturm_OptMethSoft_99}) allows us to optimize this estimation and to compute the largest positive value $\nu$ such that the linear matrix inequalities (\ref{eq:s1lmi1:lin}), (\ref{eq:s1lmi2:lin}), and (\ref{eq:s1lmi3}) 
have a solution in the variables $Q_i$ and $\nu$.

Analogous corollaries may be written by letting $\mu_i=0$, for all $i\in I$, in Theorems \ref{th:sstab2} and \ref{th:sstab4},
and by considering directly Theorem \ref{th:sstab3} (for which $\mu=0$). 

To conclude this section, let us emphasize that this approach is allowed since $\mu=0$ is possible in our main results. This is not possible using the approach of \cite{DiagneBastinCoron:automatica:11}, where $\mu$ should be a strictly positive value.

\subsection{With a diagonal source term}
\label{sub:diagonal}
If, for each $i$ in $I$, the source term $F_i$ in (\ref{eq:slin})-(\ref{eq:sbound}) is diagonal, then (\ref{eq:s1lmi1}) is equivalent to
\begin{equation}
\label{eq:s1lmi1:no:source}
-\mu \Lambda^+_i +F_i \le -\nu I_n.
\end{equation}
Therefore the following result is a corollary of Theorem \ref{th:sstab1}:
\begin{corollary}
\label{cor:diagonal}
Let us assume that for all $i\in I$, $m_i=m$, the matrices $F_i$ are diagonal and that there exist $\nu>0$, $\mu\in \RR$, and diagonal positive definite matrices $Q_i$ in $\RR^{n\times n}$ satisfying the matrix inequalities, (\ref{eq:s1lmi2}), (\ref{eq:s1lmi3}) and (\ref{eq:s1lmi1:no:source}).
Then, the switched linear hyperbolic system (\ref{eq:slin})-(\ref{eq:sbound}) is GUES
with respect to the set of switching signals $\mathcal{S}(\RR_+,I)$.
\end{corollary}

The sufficient conditions (\ref{eq:s1lmi2}), (\ref{eq:s1lmi3}) and (\ref{eq:s1lmi1:no:source})
of Corollary \ref{cor:diagonal} are nonlinear in the unknown variables $\nu$, $\mu$ and $Q_i$, due to the term $\mathcal{Q}_i(1)$ in (\ref{eq:s1lmi2}) which depends nonlinearly on 
$Q_i$ and $\mu$. However, $\mu$ being a scalar variable, one may combine a line search algorithm with semi-definite programming to solve (\ref{eq:s1lmi2}), (\ref{eq:s1lmi3}) and (\ref{eq:s1lmi1:no:source}).
Analogously, when the source terms are diagonal in (\ref{eq:slin})-(\ref{eq:sbound}), line search algorithms could be used to numerically check the sufficient conditions of Theorems \ref{th:sstab2} and \ref{th:sstab4}.

\subsection{When $\Lambda_i$ are all positive definite or all negative positive}

Let us assume in this section, that all velocities in (\ref{eq:slin})-(\ref{eq:sbound}) have the same sign, that is that either, for all $i$ in $I$, $\Lambda_i$ is positive definite or, for all $i$ in $I$, $\Lambda_i$ is negative definite. To ease the presentation of this section, it is assumed, that the first case occurs: for all $i$ in $I$, $\Lambda_i^+=\Lambda_i$ and $m_i=0$.
Then the following matrix inequalities
\begin{equation}
\label{eq:s1lmi1:positive}
-2 \mu  Q_i\Lambda_i + F_i^\top  Q_i +Q_i F_i \le -2\nu I_n,\  Q_i\le I_n,
\end{equation}
\begin{equation}
\label{eq:s1lmi2:positive}
 G_{i}^\top
  Q _i \Lambda_i  G_{i}
\leq  
   e ^{-2\mu }Q_i \Lambda_i
\end{equation}
imply (\ref{eq:s1lmi1}) and (\ref{eq:s1lmi2}).
This gives us the following corollary of Theorem \ref{th:sstab1}:
\begin{corollary}\label{cor:last}
Let us assume that, for all $i\in I$, $m_i=0$, and that there exist $\nu>0$, $\mu\in \RR$, and diagonal positive definite matrices $Q_i$ in $\RR^{n\times n}$ satisfying 
the matrix inequalities (\ref{eq:s1lmi3}), (\ref{eq:s1lmi1:positive}) and (\ref{eq:s1lmi2:positive}).
Then, the switched linear hyperbolic system (\ref{eq:slin})-(\ref{eq:sbound}) is GUES
with respect to the set of switching signals $\mathcal{S}(\RR_+,I)$.
\end{corollary}}{\startchris The sufficient stability conditions of the results presented in this paper may be solved using classical numerical tools. E.g. even if the matrix inequalities (\ref{eq:s1lmi1}) and (\ref{eq:s1lmi2}) in the statement of Theorem \ref{th:sstab1} are linear in $Q_i$ but nonlinear in $\mu$, some numerical methods may be used:
\begin{itemize}
\item by particularizing to the case $\mu=0$;
\item or when the source terms in (\ref{eq:slin})-(\ref{eq:sbound}) are diagonal;
\item or when the matrices $\Lambda_i$ are either all positive definite or all negative definite.
\end{itemize}
See \cite{Prieur:Girard:Witrant:arxiv:2013} for more details.\stopchris
}

\ifthenelse{\boolean{arxivbool}}{
The conditions (\ref{eq:s1lmi3}), (\ref{eq:s1lmi1:positive}) and (\ref{eq:s1lmi2:positive})
 of Corollary \ref{cor:last} are again nonlinear in the unknown variables $\nu$, $\mu$ and $Q_i$, due to the product of $\mu$ and $Q_i$ in (\ref{eq:s1lmi1:positive}) and the product of $e^{-2\mu}$ and $Q_i$  (\ref{eq:s1lmi2:positive}). However, since $\mu$ is a scalar variable, one may combine a line search algorithm with semi-definite programming to solve (\ref{eq:s1lmi3}), (\ref{eq:s1lmi1:positive}) and (\ref{eq:s1lmi2:positive}).
 Analogous techniques could be used as well in Theorem \ref{th:sstab2}.
For Theorem \ref{th:sstab4}, similar simplifications can be done if, for all $i\in I$,  $\Lambda_i$ is negative definite or positive definite.}{}

\section{Examples}
\label{sec:ex}

\subsection{Mode independent sign structure of characteristics}
\label{ex:A}

Consider the wave equation:
$
\partial_t^2 u(t,x)-\partial_x^2 u(t,x) = 0, 
$
where $x \in[0,1],\; t\in \RR_+$, and $u:\RR_+\times [0,1] \rightarrow \RR$.
The solutions of the previous equations can be written as $u(t,x)=w_1(t,x)+w_2(t,x)$ with 
$w=\left( \begin{smallmatrix} w_1 \\ w_2 \end{smallmatrix} \right)$ verifying
\begin{equation}
\label{eq:ex1a}
\partial_{t} w(t,x)+L\partial_{x} w(t,x) = 0, \quad  x \in[0,1],\; t \in \RR_+
\end{equation}
where $L=\diag(-1,1)$. 
We consider for this hyperbolic system the following switching boundary conditions:
\begin{equation}
\label{eq:ex1b}
w_1(t,1)=\left\{
\begin{array}{rl}
-1.2 \, w_2(t,1) & \text{if } i(t)=1\\
-0.6 \, w_2(t,1) & \text{if } i(t)=2
\end{array}
\right.,\;
w_2(t,0)=\left\{
\begin{array}{rl}
0.6 \, w_1(t,0) & \text{if } i(t)=1\\
1.2 \, w_1(t,0) & \text{if } i(t)=2
\end{array}
\right..
\end{equation}
This is a switched linear hyperbolic system of the form (\ref{eq:slin})-(\ref{eq:sbound}) with $L_1=L_2=L$,
$A_1=A_2=0_2$, $S_1=S_2=I_2$, 
$G_1=\left(  \begin{smallmatrix} 0 & -1.2 \\ 0.6 & 0 \end{smallmatrix} \right)$ and 
$G_2=\left(  \begin{smallmatrix} 0 & -0.6 \\ 1.2 & 0 \end{smallmatrix} \right)$.
With the notations defined in the previous sections, we also have
$\Lambda_1^+=\Lambda_2^+=I_2$ and $F_1=F_2=0_2$. 
\ifthenelse{\boolean{arxivbool}}{We are in the particular case described in Section~\ref{sub:diagonal}. Though, we were not able to apply Corollary~\ref{cor:diagonal} as we could not find $\nu>0$, $\mu\in \RR$, and diagonal positive definite matrices $Q_1$ 
and $Q_2$ such that the set of matrix inequalities (\ref{eq:s1lmi1:no:source}), (\ref{eq:s1lmi2}), and
(\ref{eq:s1lmi3}) 
hold.
}{In this case (\ref{eq:s1lmi1}) becomes
\begin{equation}
\label{eq:s1lmi1:no:source}
-\mu \Lambda^+_i\le -\nu I_n.
\end{equation}
We were not able to apply Theorem~\ref{th:sstab1} as we could not find $\nu>0$, $\mu\in \RR$, and diagonal positive definite matrices $Q_1$ 
and $Q_2$ such that the set of matrix inequalities (\ref{eq:s1lmi2}), 
(\ref{eq:s1lmi3}) and (\ref{eq:s1lmi1:no:source}) 
hold.
} 
Actually, this could be explained by the fact that it is possible to find a switching signal that destabilizes the system as shown on
the left part of Figure~\ref{fig:1} (where a periodic switching signal is used with a period equals to $1$). 

We can prove the exponential stability for a set of switching signals with an assumption on the average dwell time using Theorem~\ref{th:sstab2}.
Let us remark that since $F_1=F_2=0$, (\ref{eq:s2lmi1}) is equivalent to $\mu_i\ge \nu$ for $i\in \{1,2\}$. One can verify that Equations 
(\ref{eq:s2lmi2}), (\ref{eq:s2lmi3}) and  (\ref{eq:s2lmi4}) hold as well for the choices
$Q_1= \left(  \begin{smallmatrix} 0.75 & 0 \\ 0 & 2 \end{smallmatrix} \right)$, $Q_2= \left(  \begin{smallmatrix} 1.5 & 0 \\ 0 & 1 \end{smallmatrix} \right)$ and $\gamma=2$. 
Then, Theorem~\ref{th:sstab2} guarantees the stability of the switched linear hyperbolic system for switching signals with average dwell time greater than 
$\frac{\ln(\gamma)}{2\nu}=2.3105$. The right part of Figure~\ref{fig:1} shows the stable behavior of the switched linear hyperbolic system 
for a periodic switching signal with a period equals to $2.4$.
\begin{figure}[ht!]
  \begin{center}
\begin{tabular}{cc}
\includegraphics[height=\heightfigure cm]{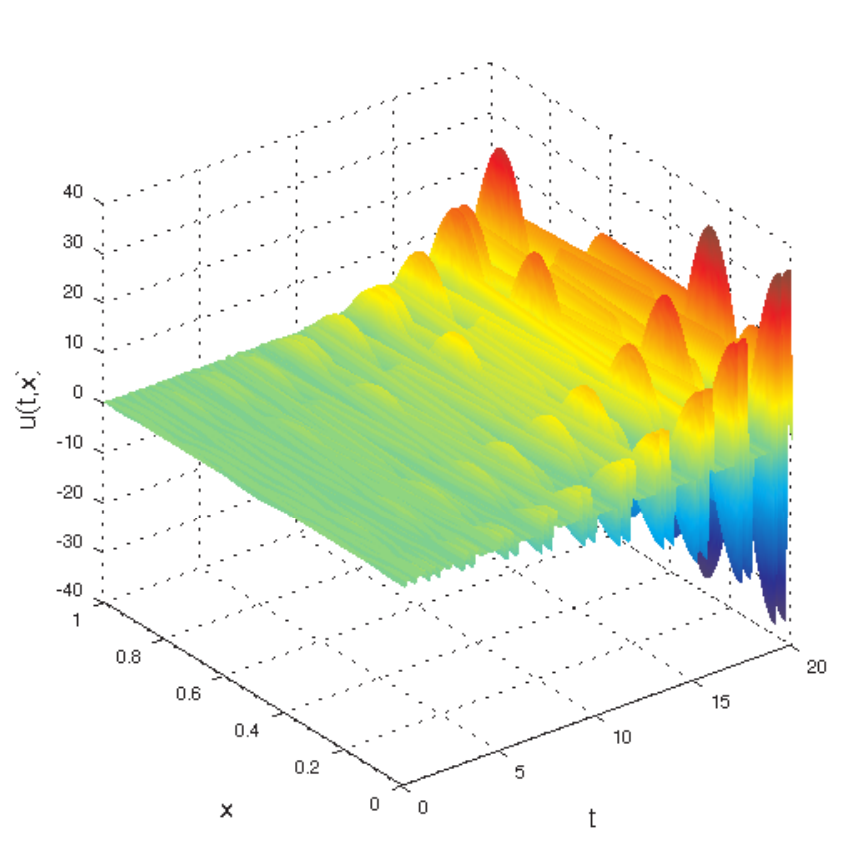}
&
\includegraphics[height=\heightfigure cm]{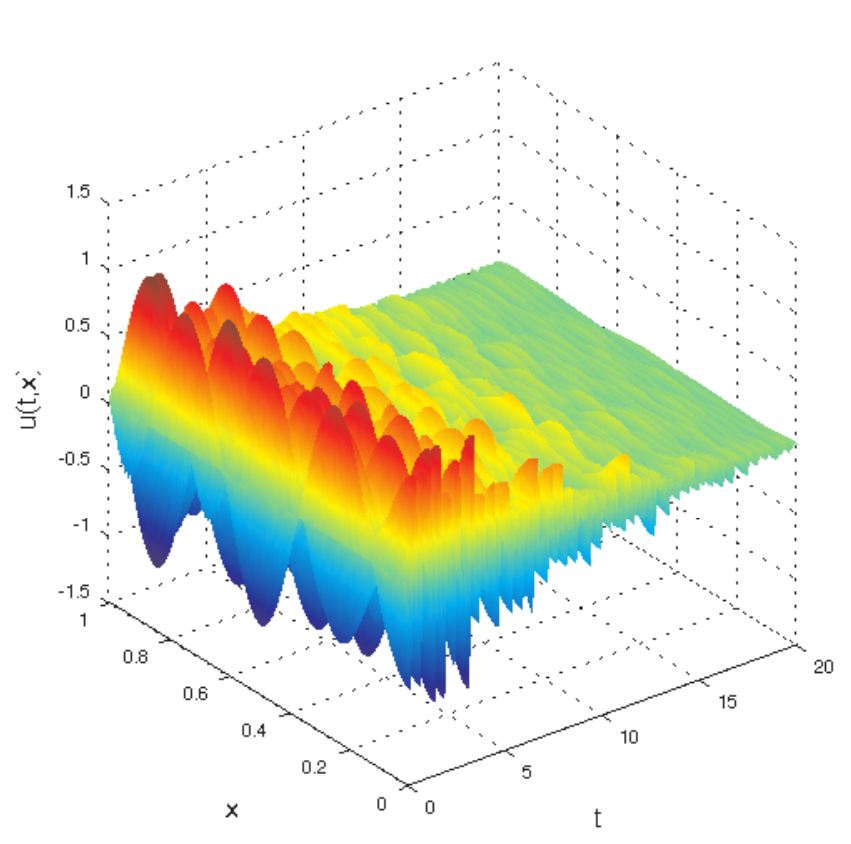}
  \end{tabular}
    \caption{Time evolution of $u=w_1+w_2$, solution of (\ref{eq:ex1a})-(\ref{eq:ex1b}) for periodic switching signals of period $1$ (left) and $2.4$ (right).}
    \label{fig:1}
  \end{center}
\end{figure}
To illustrate \ifthenelse{\boolean{arxivbool}}{Corollary~\ref{cor:diagonal}}{Theorem~\ref{th:sstab1}}, we add a diagonal damping term to (\ref{eq:ex1a}): 
\begin{equation}
\label{eq:ex1c}
\partial_{t} w(t,x)+L\partial_{x} w(t,x) = A w(t,x), \quad  x \in [0,1],\; t\in  \RR_+
\end{equation}
where $A=\diag(-0.3,-0.3)$. The boundary conditions are given by (\ref{eq:ex1b}). Now, $A_1=A_2=F_1=F_2=A$ and the other matrices of the system remain unchanged.
In the present case (\ref{eq:s1lmi1:no:source}) is equivalent to $\nu \le \mu+0.3$. 
One can verify that \ifthenelse{\boolean{arxivbool}}{Corollary~\ref{cor:diagonal}}{Theorem~\ref{th:sstab1}} applies with $\mu=-0.2$, $\nu=0.1$ and $Q_1=Q_2=\left(  \begin{smallmatrix} 1.5 & 0 \\ 0 & 1 \end{smallmatrix} \right)$. Then, Theorem~\ref{th:sstab1} guarantees the stability of the switched linear hyperbolic system for all switching signals.  Figure~\ref{fig:2} shows the stable behavior of the switched linear hyperbolic system 
for a periodic switching signal of period $1$.

\begin{figure}[ht!]
  \begin{center}
\includegraphics[height=\heightfigure cm]{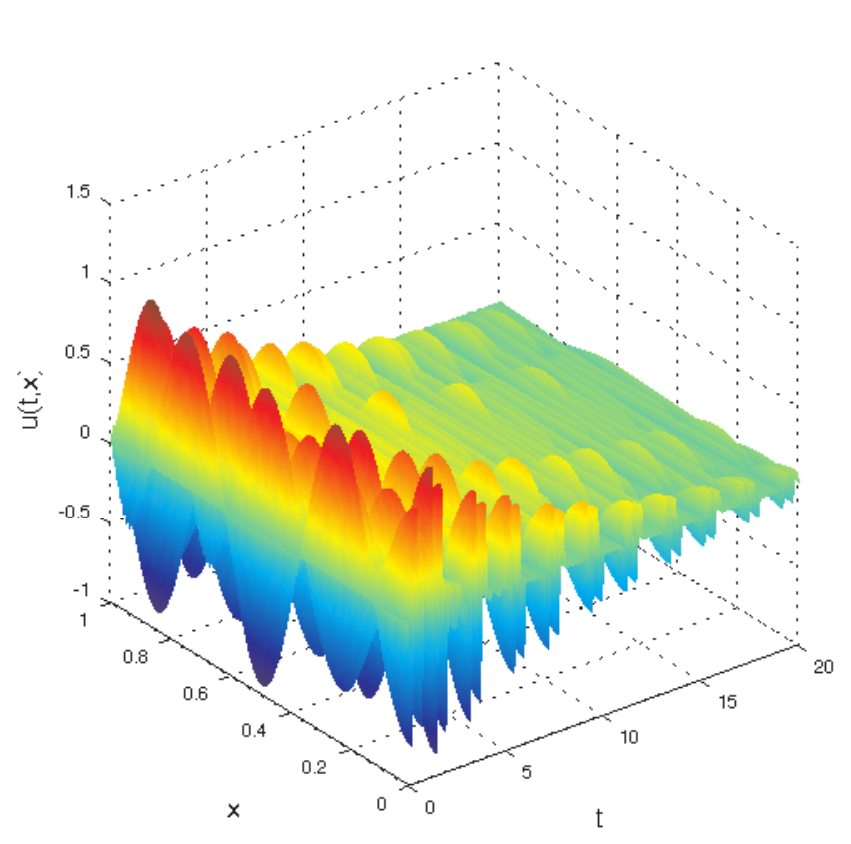}
    \caption{Time evolution of $u=w_1+w_2$, solution of (\ref{eq:ex1b})-(\ref{eq:ex1c}) for a periodic switching signal of period $1$.}
    \label{fig:2}
  \end{center}
\end{figure}

\subsection{Mode dependent sign structure of characteristics}
\label{ex:B}

To illustrate the results of Section~\ref{sec:stabedep}, we consider the following switched linear hyperbolic system 
\begin{equation}
\label{eq:ex2a}
\partial_{t} w(t,x)+L_{i(t)} \partial_{x} w(t,x) = F w(t,x), \quad  x \in[0,1],\; t\in  \RR_+
\end{equation}
where $w:\RR_+\times [0,1] \rightarrow \RR$, $i(t)\in I=\{1,2\}$, $L_1=1$ and $L_2=-1$ and $F\in \RR$. The boundary conditions are given by
\begin{equation}
\label{eq:ex2b}
\begin{array}{ll}
w(t,0)=G w(t,1) & \text{if } i(t)=1\\
w(t,1)=G w(t,0) & \text{if } i(t)=2
\end{array}
\end{equation}
and $G>0$.
This is a switched linear hyperbolic system of the form (\ref{eq:slin})-(\ref{eq:sbound}) with 
$A_1=A_2=F$, $S_1=S_2=1$, and
$G_1=G_2=G$. With the notation defined in the previous sections, we also have
$\Lambda_1^+=\Lambda_2^+=1$ and $F_1=F_2=F$. 

We assume that $F<-\ln(G)$; if this does not hold, then it can be shown that the linear hyperbolic systems in each mode are both not asymptotically stable.
If $F<0$ and $G\le 1$, then Theorem~\ref{th:sstab3} applies with $Q_1=Q_2=1$ and $\nu=-F$. Hence, in that case Theorem~\ref{th:sstab3} guarantees the stability of the switched linear hyperbolic system for all switching signals. 
If $G>1$ (resp. $F>0$) then the condition (\ref{eq:s3lmi2}) (resp. (\ref{eq:s3lmi1})) of Theorem~\ref{th:sstab3} does not hold and thus Theorem~\ref{th:sstab3} does not apply.

If $G>1$, let $F<\mu<-\ln(G)$,
then Theorem~\ref{th:sstab4} holds with
$\mu_1=\mu_2=\mu$, $\nu=\mu-F$, $\gamma=1$ and $Q_1=Q_2=1$. Then, Theorem~\ref{th:sstab4} guarantees the stability of the switched linear hyperbolic system for switching signals with average dwell time $\tau_D$ greater than $\frac{\bar \Delta_\mu}{\nu}=\frac{-2\mu}{\mu-F}$ for any $\mu \in (F,-\ln(G))$.
The minimal value of $\frac{-2\mu}{\mu-F}$ in this interval is $\frac{-2\ln(G)}{\ln(G)+F}$; therefore the
stability of the switched linear hyperbolic system is guaranteed for switching signals with $\tau_D$ greater than $\frac{-2\ln(G)}{\ln(G)+F}$. For $G=2$ and $F=-1$, in that case the minimal required $\tau_D$ is $4.5178$.
Figure~\ref{fig:3} shows unstable and stable behaviors for these values of $G$ and $F$ and for periods equal to $1.2$ and $4.6$.

If $F>0$, let $G$ and $\mu$ such that $F<\mu<-\ln(G)$,
then Theorem~\ref{th:sstab4} holds with
$\mu_1=\mu_2=\mu$, $\nu=\mu-F$, $\gamma=1$ and $Q_1=Q_2=1$. Then, Theorem~\ref{th:sstab4} guarantees the stability of the switched linear hyperbolic system for switching signals with $\tau_D$ greater than $\frac{\bar \Delta_\mu}{\nu}=\frac{2\mu}{\mu-F}$ for any $\mu \in (F,-\ln(G))$.
The minimal value of $\frac{2\mu}{\mu-F}$ in this interval is $\frac{2\ln(G)}{\ln(G)+F}$; therefore stability of the switched linear hyperbolic system is guaranteed for switching signals with $\tau_D>\frac{2\ln(G)}{\ln(G)+F}$.
For $G=0.5$ and $F=-0.1$, the minimal required $\tau_D$ is $2.3372$.
Figure~\ref{fig:4} shows unstable and stable behaviors for these values of $G$ and $F$ and for periods equal $0.9$ and $2.4$.

\begin{figure}[ht!]
  \begin{center}
\begin{tabular}{cc}
\includegraphics[height=\heightfigure cm]{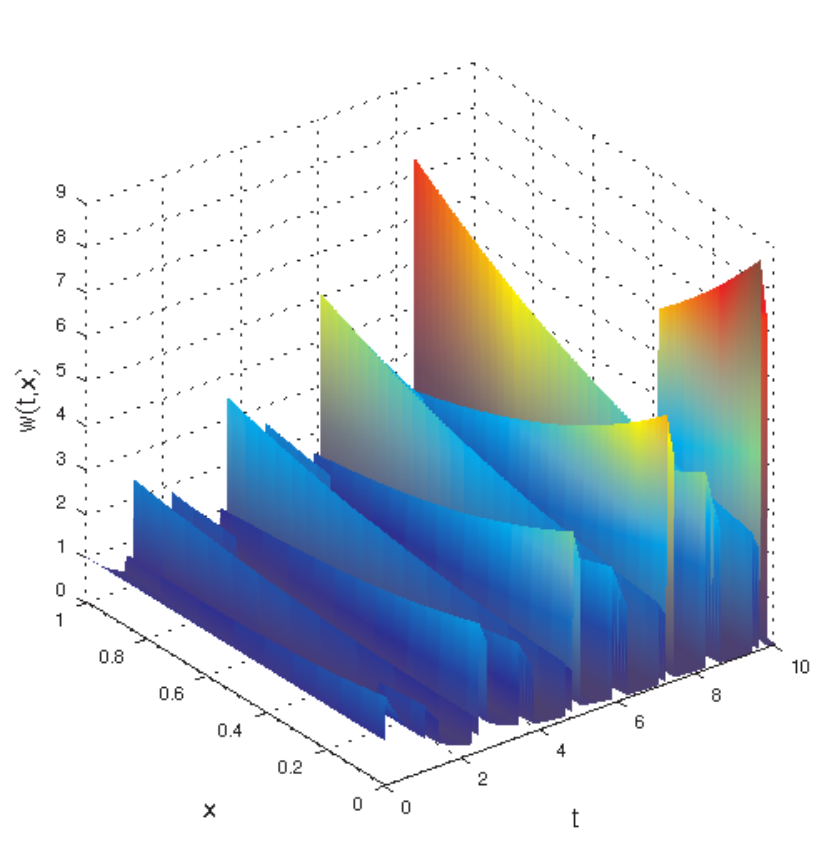}
&
\includegraphics[height=\heightfigure cm]{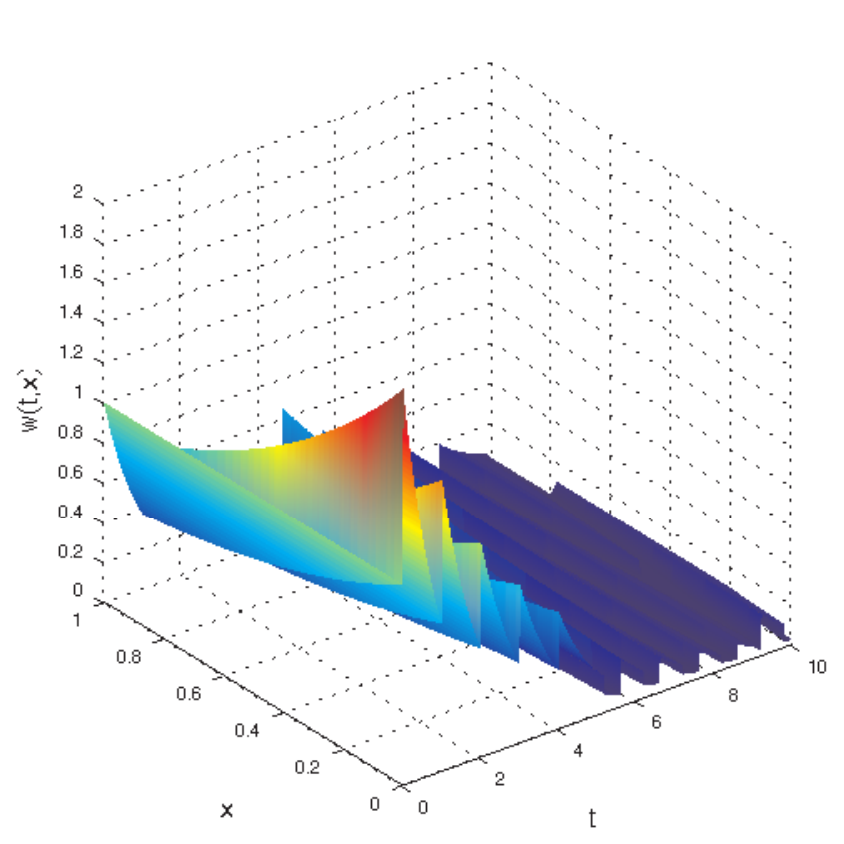}
  \end{tabular}
    \caption{Time evolution of $w$, solution of (\ref{eq:ex2a})-(\ref{eq:ex2b}) with $F=-1$ and $G=2$, for periodic switching signals of period $1.2$ (left) and $4.6$ (right).}
    \label{fig:3}
  \end{center}
\end{figure}
\begin{figure}[ht!]
  \begin{center}
\begin{tabular}{cc}
\includegraphics[height=\heightfigure cm]{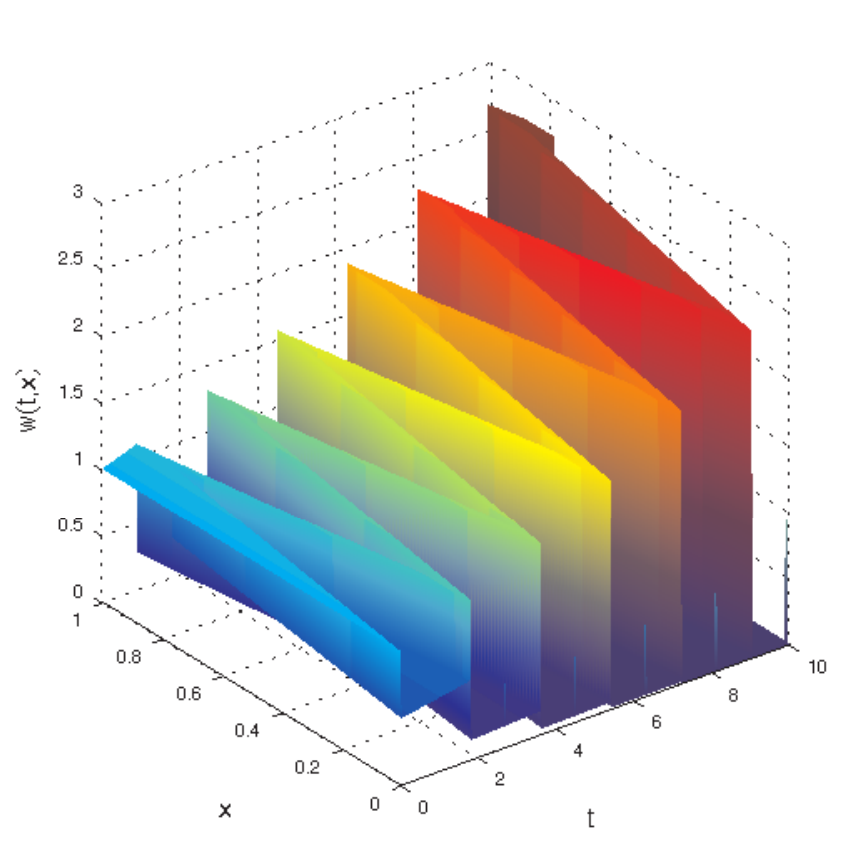}
&
\includegraphics[height=\heightfigure cm]{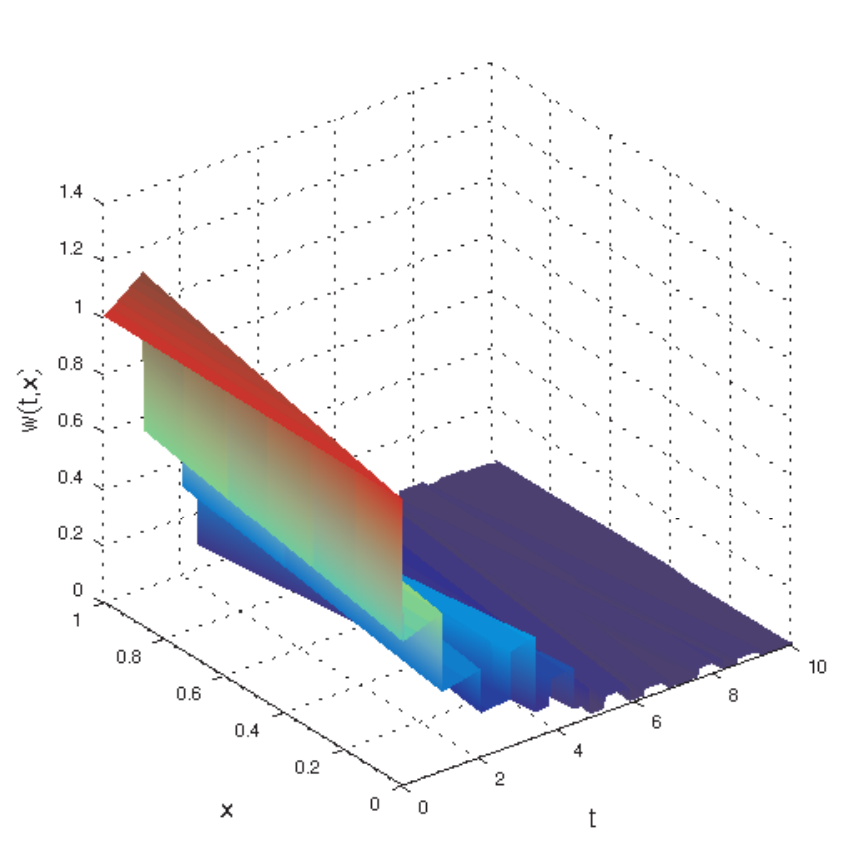}
  \end{tabular}
    \caption{Time evolution of $w$, solution of (\ref{eq:ex2a})-(\ref{eq:ex2b}) with $F=0.1$ and $G=0.5$, for periodic switching signals of period $0.9$ (left) and $2.4$ (right).}
    \label{fig:4}
  \end{center}
\end{figure}

\section{Conclusion}
\label{concl}

In this paper, some sufficient conditions have been derived for the exponential stability of hyperbolic PDE with switching signals defining the 
dynamics and the boundary conditions. 
This stability analysis has been done with Lyapunov functions and exploiting the dwell time assumption, if it holds, of the switching signals. The sufficient stability conditions are written in terms of matrix inequalities which lead to numerically tractable problems. 

This work lets many questions open and may have natural applications on physical applications. In particular, exploiting the sufficient conditions for the derivation of switching stabilizing boundary controls (as for the physical application considered in \cite{SantosPrieur:IEEE:08}) seems to be a natural extension. \startchris The generalization of the results to linear hyperbolic with space-varying entries may also be studied.\stopchris

\bibliographystyle{plain}
\bibliography{cp}
\end{document}